\documentclass[12pt]{amsart}
\usepackage{amscd}
\usepackage{verbatim}
\usepackage{amssymb, amsmath, amsthm, amscd,ifthen}
\usepackage[dvips]{graphics}
\usepackage[cp866]{inputenc}
\usepackage{graphicx}
\usepackage{epsfig}
\usepackage{color}

\usepackage{amsmath,amssymb,amscd,}
\usepackage[all]{xy}

\usepackage{amssymb}

\newtheorem{thm}{Theorem}[section]

\newtheorem{prop}[thm]{Proposition}

\newtheorem{Ex}[thm]{Example}

\newtheorem{lemma}[thm]{Lemma}

\theoremstyle{definition}

\newtheorem{dfn}[thm]{Definition}

\title[ Discrete Morse theory for  flexible polygons]{Discrete Morse theory for  moduli spaces of   flexible polygons, or solitaire  game on the circle}
\author{Gaiane Panina and Alena Zhukova }

\address{
  }


\begin{document}
\begin{abstract}
We introduce a perfect  discrete Morse function on  the moduli
space of a  polygonal linkage.
The ingredients of the construction are: (1) the cell structure on
the moduli space, and (2) the discrete Morse theory approach, which
gives a way to reduce the number of cells to the minimal possible.

\end{abstract}

 \maketitle \setcounter{section}{0}

\section{Introduction}\label{section_abstr}

A Morse function on a smooth manifold is called \textit{perfect} if the number of critical points equals the sum of Betti numbers. Analogously, a discrete Morse function on a cell complex is called \textit{perfect} if the number of critical cells equals the sum of Betti numbers\footnote{In the paper we always assume that  homology groups and Betti numbers are with coefficients in $\mathbb{Z}$.}.
In a sense, a perfect Morse function (either smooth or discrete) is the optimal one:
all the Morse inequalities turn to equalities, the critical points (critical cells, in discrete framework) represent independent generators of
the homology groups, and therefore,
the number of  critical points (critical cells)
is the minimal possible.
Not every
 manifold admits a perfect Morse functions.
Even if it exists, it is generically  hard to find it. In the discrete setting, it is an NP-hard problem,
see \cite{Jos, Burton}.

In the present paper we explicitly build a perfect  discrete Morse function on  the moduli
space of a   polygonal linkage.

The starting point of our construction is  a cell
decomposition  of the moduli
space constructed in \cite{pan2} and reviewed in the next section. The number of cells is big:
it exceeds the sum of Betti numbers very much. Following R. Forman,
we introduce a discrete Morse function on the cell complex which turns
to be  perfect. According to the discrete Morse
theory, this gives a way of contracting some of the cells such that
the number of remaining cells is the minimal possible. The rules of
manipulating the cells, and the rules describing gradient paths
resemble the solitaire  game. However, this analogy should not be
taken too seriously: it is   a mere metaphor, not a mathematical
statement.

The perfect Morse function is constructed in two steps. On the first
step we introduce some natural pairing on the cell complex which
substantially reduces the number of critical cells.  However the
number of critical cells is not yet minimal possible.

On the second step (following once again R. Forman) we apply \textit{path
reversing technique}, which gives a new Morse function.
This technique is the discrete version of the Milnor-Smale ''First Cancelation Theorem'', see \cite{Milnor}.
Originally it allows to reverse just one gradient path, whereas we reverse many of them at a time. In our
particular case a careful choice of the paths to be reversed yields a perfect discrete Morse function.
It is worthy to mention that the idea of simultaneous reversal  of several gradient paths is not new: it appeared in P. Hersh's paper  \cite{Hersh}.

\medskip

Using our approach, it is possible to compute  homology groups of
the configuration space of a polygonal linkage independently on the
proof of M. Farber and D. Sch\"{u}tz  \cite{faS}. However, such a
proof does not seem to be a short one, so we do not give the details
here.

To the best of our knowledge, no smooth perfect Morse function on
the moduli space   of a  polygonal linkage is known.  This motivates
us to formulate the following {open problem}:

\textit{What is the smooth counterpart of the proposed discrete
Morse function?} We mean here not an existence-type theorem, but a function expressed by some (possibly short)
formula and having some transparent physical or geometrical meaning.

The other question is:

\textit{ Is there a similarity  
with the approaches of E. Babson and P. Hersh \cite{BH, Hersh}?}
\bigskip

\textbf{Acknowledgements.} The present research  is  supported by
RFBR  project No. 15-01-02021. We also express our gratitude to all
the participants of the seminar ''Geometry and combinatorics'' of
Chebyshev Laboratory of St. Petersburg State University for
inspiring discussions. In particular, we are much indebted to Pavel Galashin for producing the
starting  idea of the  pairing.

\section{Preliminaries}\label{section_prelim}

We start with two necessary remindings.
\subsection*{Polygonal linkage: moduli space and the cell complex \cite{pan2}}

A \textit{polygonal  $n$-linkage} is a sequence of positive numbers
$L=(l_1,\dots ,l_n)$. It should be interpreted as a collection of
rigid bars of lengths $l_i$ joined consecutively in a chain by
revolving joints. We always assume that the triangle inequality
holds, that is, $$\forall j, \ \ \ l_j< \frac{1}{2}\sum_{i=1}^n
l_i$$ which guarantees that the chain of bars can close.

 \textit{A planar configuration} of $L$   is a sequence of points
$$P=(p_1,\dots,p_{n}), \ p_i \in \mathbb{R}^2$$ with
$l_i=|p_i,p_{i+1}|$, and $l_n=|p_n,p_{1}|$. We also call $P$  a
\textit{polygon}.

As follows from the definition, a configuration may have
self-intersections and/or self-overlappings.

\begin{dfn}\label{defConfSpace} \textit{The moduli space, or the
configuration space\footnote{Also known as \textit{polygon space}.} $M(L)$}  is the set  of all configurations
modulo orientation preserving isometries of $\mathbb{R}^2$.

Equivalently, we can define $M(L)$ as
$$M(L)=\{(u_1,...,u_n) \in (S^1)^n : \sum_{i=1}^n l_iu_i=0\}/SO(2).$$
\end{dfn}

The latter definition shows that $M(L)$ does not depend on the
ordering of $\{l_1,...,l_n\}$; however, it does depend on the values
of $l_i$.

Throughout the paper  we assume that no configuration of $L$ fits a
straight line. This assumption implies that the moduli space $M(L)$
is a closed $(n-3)$-dimensional smooth manifold. Informally, the dimension of the manifold means  that a polygonal linkage is flexible with degree of freedom  $n-3$.
Smoothness comes in a more tricky way, from Morse theory on the configuration space of \textit{a robot arm}\footnote{It is also possible to prove that $M(L)$ is a manifold by  using some of the angles of a configuration as local coordinates.}, see \cite{F}

The manifolds $M(L)$ appear naturally in topological robotics and are well studied: their homology groups are known \cite{faS},
the Walker conjecture (on retrieving the edge lengths $l_i$ from the cohomology ring of $M(L)$) has been discussed \cite{Walker}, Morse theory has been applied \cite{F, pkh}. However, the existence of a perfect Morse function has not been  established.

\bigskip

An important ingredient of our construction is  the explicit combinatorial description of
$M(L)$ as a  regular  cell complex $\mathcal{K}(L)$.
We first give some notation.
A subset $I$ of $[n]=\{1,2,...,n\}$  is \textit{short} if $$\sum_{i\in I}
l_i<\frac{1}{2}\sum_{i=1}^nl_i.$$

A partition of $[n]=\{1,2,...,n\}$ is called \textit{admissible} if
all the parts are short.

Given a partition, the set containing the entry ''$n$''  is called\textit{ the
$n$-set};
a\textit{ singleton} is a set containing exactly one entry.

\bigskip

 {\bf A remark on notation for a cyclically ordered partition.} For a partition of $[n]$, the $n$\textit{-set }is the set containing the entry $n$.
 We write a cyclically ordered partition of $[n]$
as a (linearly ordered) string of sets where the $n$-set  stands on
the last position.
We stress  that for an ordered partition, the order of the sets matters, whereas
there is no ordering inside a set. For example,
$$\big(\{1\} \{3 \} \{4,  2, 5,6\}\big)\neq\big(\{3 \}\{1\}  \{4,  2, 5,6\}\big)= \big( \{3 \}\{1\}\{ 2,4, 5,6\}\big).$$

Before we describe the cell complex, remind that a CW-complex can be
constructed inductively by defining its skeleta. Once the $(k -
1)$-skeleton is constructed, we attach a collection of closed
$k$-balls $C_i$ by some continuous mappings $\varphi_i$ from their
boundaries $\partial C_i$ to the $(k-1)$-skeleton. For a
\textit{regular } complex, each of  the mappings $\varphi_i$ is
injective, and $\varphi_i$ maps $\partial C_i$ to a subcomplex of
the  $(k-1)$-skeleton. Regularity of a complex implies that a
complex is uniquely defined by the poset of its cells. Regularity
also guarantees the existence of well-defined barycentric
subdivision and (for manifolds) the well-defined dual complex.

\begin{thm}\label{MainThm}
We have  a structure of a regular CW-complex $\mathcal{K}(L)$ on the
moduli space $M(L)$. Its complete combinatorial description reads as
follows:
\begin{enumerate}
    \item  $k$-dimensional cells of the complex $\mathcal{K}(L)$
    are labeled by cyclically ordered admissible partitions of
 the set  $[n]=\{1,2,...,n\}$  into $(n-k)$ non-empty
parts.

    \item A closed cell $C$ belongs to the boundary of some other closed cell
    $C'$  iff  the partition  $\lambda(C)$ is finer than
    $\lambda(C')$.\qed

\end{enumerate}
\end{thm}

\begin{figure}[h]
\centering
\includegraphics[width=10 cm]{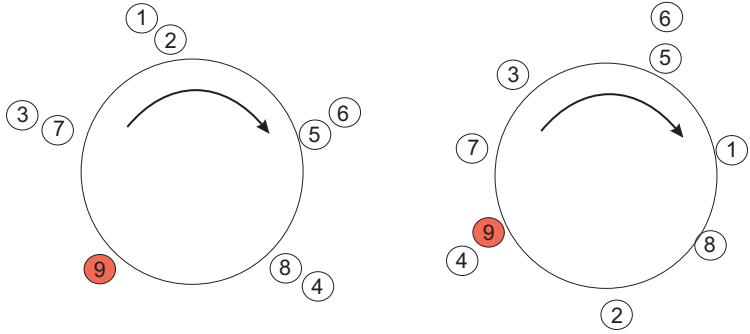}
\caption{A $4$-cell and a $2$-cell. We write these labels as
$\big(\{3,7\}\{1,2\}\{5,6\}\{4,8\}\{9\}\big)$ and
$\big(\{7\}\{3\}\{5,6\}\{1\}\{8\}\{2\}\{4,9\}\big)$} \label{Figcells}
\end{figure}

In the sequel, instead of saying ''the cell of the complex labeled
by $\lambda$'' we say for short ''\textit{the cell} $\lambda$''.

 Given a cell $\lambda$, its facets are obtained by splitting one of
 the parts of the partition $\lambda$ into two non-empty parts.
 For example, the cells $$\big(\{7\}\{3\}\{1,2\}\{5,6\}\{4,8\}\{9\}\big)\hbox{ and }\big(\{3\}\{7\}\{1,2\}\{5,6\}\{4,8\}\{9\}\big)$$
 are facets of the cell $\big(\{3,7\}\{1,2\}\{5,6\}\{4,8\}\{9\}\big)$

\medskip

Let us explain in some more details how the cell structure appears.
We start by putting \textit{labels} on the elements of the configuration space:
according to the Definition \ref{defConfSpace}, each configuration
is a collection of unit vectors $\{u_i\}$.  If the vectors are
different they give a cyclic ordering on the set $[n]$. If some of
the vectors coincide, there arises a cyclically ordered partition of $[n]$,
whose parts correspond to sets of coinciding vectors. By triangle inequality   all
the labels are admissible partitions. Conversely, each admissible partition arises in this way.

Next, we introduce equivalence classes: two points from $M(L)$ (that is, two configurations) are
\textit{equivalent} if they have one and the same label. Equivalence
classes  of $M(L)$   are the \textit{open cells}. The closure of an
open cell (taken in  $M(L)$) is called a \textit{closed cell}; it is homeomorphic to a ball. For a cell
$C$, either closed or open, its label $\lambda (C)$ is defined as
the label of (any) its interior point. The collection of open cells
yields  a structure of a regular CW-complex which is dual to the
complex  $\mathcal{K}(L)$.

\subsection*{Discrete Morse function on a regular cell complex \cite{Forman}}
Assume we have a regular cell complex. By $\alpha^p, \ \beta^p$ we
denote its $p$-dimensional cells, or \textit{$p$-cells}, for short.

 A \textit{discrete vector field} is a set of pairs
$\big(\alpha^p,\beta^{p+1}\big)$
 such that:
\begin{enumerate}
    \item  each cell of the complex participates in at most one
    pair, and
    \item  in each pair, the cell $\alpha^p$ is a facet of $\beta^{p+1}$.

\end{enumerate}

Given a discrete vector field, a \textit{path}  is a sequence of
cells

$$\alpha_0^p, \ \beta_0^{p+1},\ \alpha_1^p,\ \beta_1^{p+1}, \ \alpha_2^p,\ \beta_2^{p+1} ,..., \alpha_m^p,\ \beta_m^{p+1},\ \alpha_{m+1}^p,$$
which satisfies the conditions:
\begin{enumerate}
    \item  Each $\big(\alpha_i^p,\ \beta_i^{p+1}\big)$ is a pair;
    \item
    $\alpha_i^p$ is a facet of $\beta_{i-1}^{p+1}$;
    \item $\alpha_i\neq \alpha_{i+1}$.
\end{enumerate}

A path is  \textit{closed } if $\alpha_{m+1}^p=\alpha_{0}^p$.
A \textit{discrete Morse function on a regular cell complex} is a
discrete vector field without closed paths.

Assuming that a discrete Morse function is fixed, the \textit{critical cells} are those cells of the complex that
  are
not paired. Morse inequality says that we cannot avoid them
completely; our goal is to minimize their number.

A \textit{ gradient path} of a discrete Morse function leading from
one critical cell $\beta^{p+1}$ to some other  critical cell
$\alpha^{p}$ is a sequence of cells satisfying the three above conditions:
$$\beta^{p+1}=\beta_0^{p+1},\ \alpha_1^p, \ \beta_1^{p+1},\ \alpha_2^p,\ \beta_2^{p+1},\ \alpha_3^p,\ \beta_3^{p+1},\ ...,\ \alpha_m^p,\ \beta_m^{p+1},\ \alpha_{m+1}^{p}=\alpha^{p}$$

A discrete Morse function is a  \textit{perfect Morse function}
whenever the number of critical $k$-cells  equals the $k$-th Betty
number of the complex.  It is equivalent to the condition that the
number of all critical cells equals the sum of Betty numbers.

\section{ Pairing on the complex $\mathcal{K}$: ''rules of the game''.}\label{section_pairings}
Assume that a linkage $L=(l_1,...,l_n)$ is fixed. Without loss of
generality we may assume that

$$l_n \geq l_{n-1} \geq ... \geq l_1.$$

First we give some \textbf{notation}:

\begin{enumerate}
    \item By ''$\cdots$'' we denote any (possibly empty) ordered admissible collection of subsets of $[n]$.
\item By ''$*$'' we denote any  (possibly empty) subset of $[n]$.
    \item A set $I\subset [n]$ is $k$-\textit{prelong}, if
 $I$ is short, and  $I\cup\{k\}$ is long.
 \item For a set $I\subset [n]$ and  $k\in [n]$, we write
 $k<I$
 whenever $\forall i\in I, \ \ k<i$.
\item Analogously, we write
$k=Min(I)$
 whenever $k$ is the minimal entry of the set $I$.

\end{enumerate}

\bigskip

Now we describe a  discrete Morse function. The rules of pairing are illustrated in Figure \ref{FigPair}.

\subsection*{Step 1}
We pair together
$$\alpha=\big(\cdots\ \{1\}\ I\ \cdots\big)  \hbox{ and } \beta=\big(\cdots\ \{1\}\cup I\ \cdots\big)$$ iff the following holds:
\begin{enumerate}
    \item  $n \notin I$, and
    \item the set $\{1\}\cup I$  is short.
\end{enumerate}

Before we pass to step 2, observe that the non-paired cells are
labeled  by one of the following types of labels:
$$\big(\cdots\  \{n,1,*\}\big),$$
$$ \big(\cdots \ \{1\}\  \{n,*\}\big)\big),$$
$$\big(\cdots \ \{1\}\ \big(\hbox{a 1-prelong set}\big)\ \cdots\big).$$

\subsection*{Step 2}
We pair together
$$\alpha=\big(\cdots\ \{2\} \ I\ \cdots\big)  \hbox{ and } \beta=\big(\cdots\ \{2\}\cup I\ \cdots\big)$$ iff the following holds:
\begin{enumerate}
    \item $1 \notin I, \ 2 \notin I.$
    \item The set $\{2\}\cup I$  is short.
    \item $\alpha$ and $\beta$ were not paired at the previous step.
\end{enumerate}

We proceed this way for all $k<n$, assuming that the step number $k$
looks as follows:
\subsection*{Step k}
We pair together
$$\alpha=\big(\cdots\ \{k\}\ I\ \cdots\big)  \hbox{ and } \beta=\big(\cdots\ \{k\}\cup I\ \cdots\big)$$ iff the following holds:
\begin{enumerate}
    \item $n \notin I, \ 1 \notin I, \ 2 \notin I, ..., \ (k-1) \notin I.$
    \item $\alpha$ and $\beta$ were not paired at the previous steps.
\end{enumerate}

\begin{figure}[h]
\centering
\includegraphics[width=10 cm]{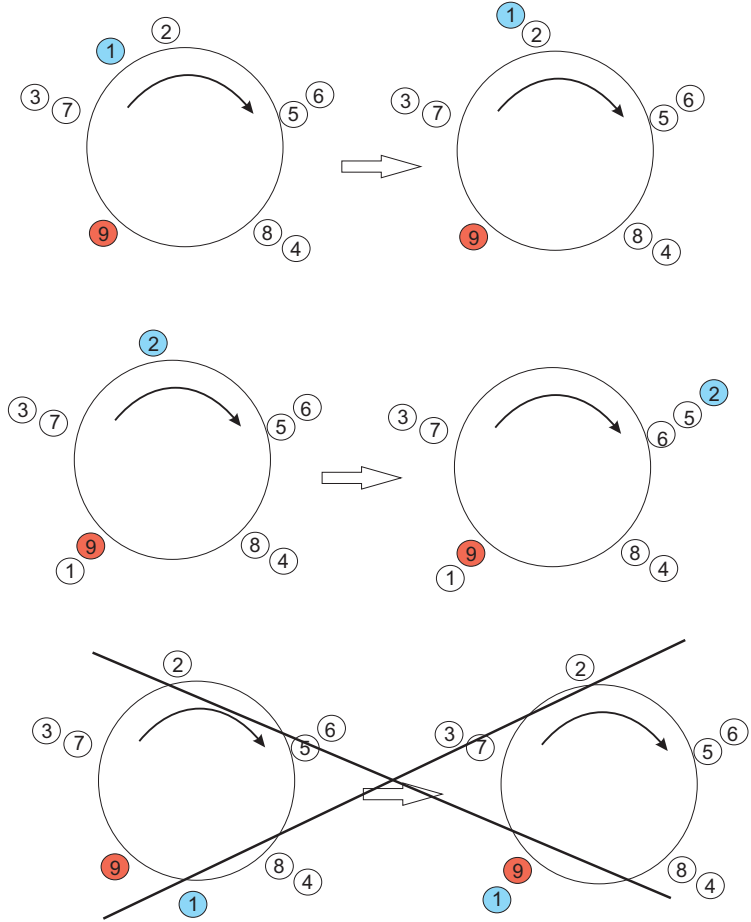}
\caption{Pairing in the complex: examples and
non-example. Take a polygon with $n=9$ and assume that all the sets depicted here are short. The first pairing comes on the  step 1. The second pairing comes on step 2.
The cells at the bottom do not form a pair: by our rules, no entry can enter the $n$-set.}\label{FigPair}
\end{figure}

\bigskip

\subsection*{Pair search algorithm}

It is convenient to have an algorithm that finds a pair (if there is one)
 for a given cell $\alpha$. The algorithm will be a useful tool for finding
 gradient paths.

First observe that if two cells are paired, they differ
  by moving one entry either inside or outside one of the sets. Observe also  that no pairing changes the $n$-set.

An entry  $k\neq n$ is \textit{forward-movable} with respect to the cell
$\alpha$ if it forms a singleton in this cell followed by a set $I$, $n \notin I$ such that
\begin{enumerate}
    \item $k<I$, and
    \item $\{k\}\cup I$ is short.

\end{enumerate}

An  entry $k$ is \textit{backward-movable} if the following holds:
\begin{enumerate}
\item the entry $k$ lies in a non-singleton set $J$, $n \notin J$;
\item $k=Min(J)$;
\item one of the following conditions holds: \begin{enumerate}
\item the set $J$ is preceded by  a non-singleton set;
\item the set $J$ is preceded by  a singleton $\{m\}$ with $m>k$;
\item the set $J$ is preceded by  the $n$-set.
\end{enumerate}\end{enumerate}

In this notation, the \textbf{algorithm} looks as follows:

Given a cell $\alpha$, take the minimal  movable entry $k$ in
$\alpha$. Then the cell $\alpha$ is paired on the step number $k$ with a cell that
is formed from $\alpha $ by moving $k$ either forward or backward.

An immediate corollary of the pairing construction is the following lemma whose informal message is: along a gradient path, ''small'' entries move forward whereas ''big'' entries move backward.
\begin{prop}\label{LemmaChangeOrder}

\begin{enumerate}
  \item Assume we have a gradient path of the described above discrete vector field.
Assume also that $m<k$, and a cell $$\alpha=\big(\cdots\ \{k,*\}\ \cdots\ \{m,*\}\ \cdots\big)$$ belongs to the path
(that is, the entries  $k$ and $m$ belong   to different sets, and  the entry  $k$ is placed to the
left of the entry $m$).

 Then during the gradient path after the cell $\alpha$, the entries $k$ and $m$ never get in one and the same set and never change their
 order.
  \item The introduced discrete vector field is a discrete Morse function.
\end{enumerate}

\end{prop}

\begin{proof}  (1) follows from the pairing construction. (2). In a closed gradient
path at least two
entries interchange their order during the path, which contradicts (1).
 \end{proof}

\section{Critical cells of the complex $\mathcal{K}$}
 Let us list all the critical
cells (that is, the cells that are non-paired). They are exactly
those  with empty set of movable entries.

\bigskip

\textbf{Notation}: unlike ''$\cdots$'',  by ''$\spadesuit$'' and
''$\clubsuit$'' we denote  a (possibly empty) string of singletons
going in the decreasing order.  For instance, ''$\spadesuit$'' can be
$\big(\{7\}\{5\}\{3\}\big)$ but  neither
$\big(\{7,5,3\}\big)$ nor $\big(\{5\}\{3\}\{7\}\big).$

We first give examples and next formulate  a  theorem.

\textbf{Examples of critical cells:}
\begin{enumerate}
\item $\big(\{7\}\{5\}\{3\}\{8,1,2,4,6\}\big) $  is a critical cell.

\item $\big(\{5\}\{3\}\{6,4\}\{1\}\{7,2\}\big)$ is a critical cell
assuming that $\{6,4\}$ is $3$-prelong.
\end{enumerate}

\textbf{Non-examples: }

\begin{enumerate}
\item The cell $\big(\{7,5\}\{3\}\{8,1,2,4,6\}\big)$ is non-critical because it
is paired with $\big(\{5\}\{7\}\{3\}\{8,1,2,4,6\}\big).$ Here $5$ is
a movable entry.
\item The cell $\big(\{5\}\{6\}\{3\}\{2\}\{1\}\{8,4,7\}\big)$ is non-critical because it is paired with
$\big(\{5,6\}\{3\}\{2\}\{1\}\{8,4,7\}\big).$ Here singletons do
not come in decreasing order, $5$ is a movable entry.

\item The cell $\big(\{7\}\{5\}\{3\}\{6,2\}\{1\}\{8,4\}\big)$
 is also non-critical. It is paired with
$\big(\{7\}\{5\}\{3\}\{2\}\{6\}\{1\}\{8,4\}\big).$
\end{enumerate}

\begin{thm} The critical cells of the introduced above discrete Morse function are exactly all cells of the two following types illustrated in Figure \ref{FigCritCells}:
\end{thm}

Cells of \textbf{type 1} are labeled by
$$\big(\spadesuit \ \{n,*\}\big).$$

\bigskip

Cells of \textbf{type 2} are labeled by

$$\big(\spadesuit \ \{k\}\  I \ \clubsuit \ \{n,*\}\big),\hbox{ if the following conditions hold:}$$

\begin{enumerate}
    \item $I$ is a $k$-prelong set.
\item $k< I$.
    \item $k< \spadesuit$.

\end{enumerate}

\bigskip

\begin{proof}

Clearly, all the above cells  have no movable entries and therefore
are critical. To prove the converse, consider two cases for a
critical cell $\alpha$:
\begin{enumerate}

\item The partition  $\alpha$ consists only of singletons.
Then the singletons  necessarily go in decreasing order, otherwise
there exists a forward-movable entry. Thus we get a  critical cell
of type $1$.

\item The partition $\alpha$ contains some non-singleton sets.
Each non-singleton is either a prelong set (with respect to its
preceding entry), or the $n$-set; otherwise a simple case analysis
shows the existence of a movable entry.
\end{enumerate}
\end{proof}

\begin{figure}[h]
\centering
\includegraphics[width=10 cm]{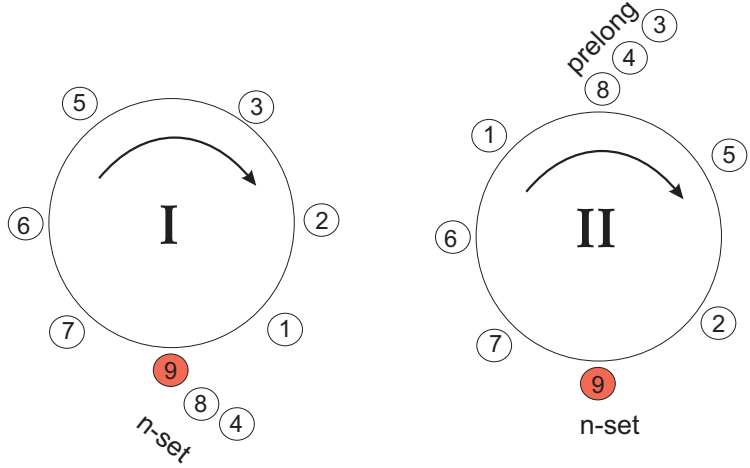}
\caption{Critical cells for $n=9$. We assume that  $\{8,4,3\}$ is $1$-prelong}\label{FigCritCells}
\end{figure}

\begin{Ex}Assume that $L=(1,1,...,1,(n-1-\varepsilon))$. In this case the configuration space $M(L)$ is known to be the $(n-3)-$sphere \cite{F}.
The (only two) critical cells of the Morse function are

$$\big(\{n-1\}...\{3\}\{2\}\{1\}\{n\}\big)$$   and   $$\big(\{1\}\{n-1,...,3,2\}\{n\}\big),$$

 that is, we have a perfect Morse function for
this particular case.

\end{Ex}

\begin{Ex}
Another example when we have a perfect Morse function is given by
$L=\big(\varepsilon,\varepsilon,\varepsilon,...,\varepsilon, 1,1,1)$
The configuration space $M(L)$ equals the disjoint union of two
tori. The critical cells are labeled either by
$$\big(\{n-1\}\{n-2\}\clubsuit\{n,*\}\big), \hbox{(Type 1)}$$
or by
$$\big(\{n-2\}\{n-1\}\clubsuit\{n,*\}\big), \hbox{(Type 2)}$$

so one easily concludes that the number of critical cells of a fixed
dimension $k$ equals the Betti  number $b_k(M(L))$.
\end{Ex}

However, the above two examples are very exceptional: in other cases
the introduced Morse function is far from perfect. Rough estimates
show that the number of critical cells is much bigger than the sum
of Betti numbers.

\section{Gradient paths between critical cells}\label{SecPaths}

 According to the definition, a gradient path between critical cells   is an alternating
sequence of \textit{ join-steps} (pairing $\alpha_i^p$ and
$\beta_i^{p+1}$), and \textit{split-steps} (choosing a facet
$\alpha_{i+1}^p$ of $\beta_i^{p+1}$). A gradient path   between critical cells always starts and ends
by a split.
 A join-step
 decreases the number of sets in the partition  by one, whereas
a split-step  increases the number of sets by one.

Each  join-step is uniquely defined according to our pairing
algorithm: it is performed by moving forward the minimal movable entry.
The entry joins the consecutive
 set in the partition.

Another important remark is that if one starts a series of steps
with a cell $\beta^{p+1}$,  one does not necessarily arrive at some
critical cell $\alpha^p$. This is similar to a solitaire player, who
not always wins, but sometimes gets stuck. Below we exemplify
''successful'' solitaire games. The reader can try some other types
of splitting and work out  some loosing examples.

So we have some freedom for a split-step, but in many cases the
freedom is illusive: if we wish to reach some critical cell at the end of the path, the split-step for a gradient path often is
defined uniquely. Indeed, if after some split-step the smallest
movable entry is backward-movable, there exists no consecutive
join-step.

\medskip

 Assume we have a gradient
path from a  critical cell $\beta=\big(\spadesuit_1\
\{j_1\}\ I_1 \ \clubsuit_1 \ \{n,*_1\}\big)$ to a  critical cell
$\alpha=\big(\spadesuit_2\ \{j_2\}\  I_2 \ \clubsuit_2 \
\{n,*_2\}\big)$.
 We say
that \textit{the prelong set $I$ is maintained during the gradient path
} if each cell of the path has a set containing $I$.  In other
words, during the path, the set $I$ may accept and lose new entries,
but it may not lose its initial entries.

\begin{lemma}\label{LemmaSplitPrelong} Assume we have a gradient
path from a  critical cell \newline $\beta=\big(\spadesuit_1\
\{j_1\}\ I_1 \ \clubsuit_1 \ \{n,*_1\}\big)$ to a  critical cell
$\alpha=\big(\spadesuit_2\ \{j_2\}\  I_2 \ \clubsuit_2 \
\{n,*_2\}\big)$.

  If $I_1 \neq I_2$, then $j_1\neq j_2$, and the entry $j_2$ belongs to $*_1$.

\end{lemma}

\begin{proof}
Consider the join-step after which the set $I_2$ appears
in the path and stays maintained until the end. On this
step,  the entry $k=Min(I_2)$ joins the set $I_2\setminus \{k\}$.
 Since $k$ is the minimal  movable entry at this step, for $j_2<k$  there are only two possibilities: (1) either $j_2$ is in the $n$-set, or  (2) $j_2$ goes after $I_2$. 
 The second case is excluded, since
 no entry can pass through the $n$-set.
\end{proof}

The  lemma together with a case analysis allows us to describe the
gradient paths between critical cells. We do not present the
complete list of all possible gradient paths, since we actually do
not need all of them. The point is that in the next section we are
going to reduce the number of the critical cells using  path
reversing, and arrive at a perfect Morse function.

\begin{prop}
  There are no gradient paths from a critical cell
  of type 1 to a critical cell of type 2.
\end{prop}
Proof. Assume that there is a path leading from  the cell
\newline $\beta=\big(\spadesuit_1 \
 \{n,*_1\}\big)$ to the cell $\alpha=\big(\spadesuit_2 \ \{k\}\  I
\ \clubsuit \ \{n,*_2\}\big).$

Then by Proposition \ref{LemmaChangeOrder},(1),  not more than one singleton
$j$ from $\spadesuit_1$ belongs to $I$. Moreover, since all others
entries of $I$ eventually join it, we have $j=Max(I)$. All other
entries of $I$ and also  $k$  come from $*_1$. So we necessarily
have
$$\big(I\setminus \{Max(I)\}\big) \cup \{k\} \subseteq *_1.$$

The  set $I \cup \{k\}$ is long, therefore ${Max(I)} \cup \{
*_1\}$ is also long, which implies that $\{n, *_1\}$ is long as well. A contradiction. \qed

\bigskip

\begin{prop}\label{PropToReverseGradPath}Assume that   $$\beta =\big(\spadesuit \ \{k\} I \ \clubsuit
\{n,*,j\}\big)\  \hbox{ and } \ \alpha=\big(\spadesuit \ \{k\} I \
\clubsuit \cup\{j\} \ \{n,*\}\big)$$ are critical cells of type 2.
 If $I$ is   $j$-prelong, the cells are connected  by exactly one gradient
 path. During the path, the entry $j$ splits from the $n$-set backward, and  joins
$\clubsuit$,  see Fig. \ref{FigNewRules} for an example.
\end{prop}

\begin{proof}
 We search for possible paths from $\beta$ to $\alpha$. By Lemma \ref{LemmaSplitPrelong}, these paths do not contain splits of the prelong set. So the path starts with the split of the
  $n$-set. We easily conclude that the entry $j$  spits backward.
\end{proof}


\section{Path reversing: new discrete Morse function}

Our next step  is to reduce the number of critical cells
using the following theorem:

\begin{thm}\cite{Forman} Suppose we have a discrete Morse function  with critical
cells  $\alpha$, $\beta$  such that there exists exactly one
gradient path from  $\beta$ to $\alpha$. Then reversing the
direction of this gradient path produces a discrete Morse function
with $\alpha$, $\beta$ no longer critical.\qed
\end{thm}

A necessary warning is: such paths should be reversed one by one,
since reversing one path may create new paths between other
pairs of critical cells. Keeping this in mind, we do not reverse all
the paths that are described in Proposition
\ref{PropToReverseGradPath}, but pose some extra condition.

\bigskip

\subsection*{Path reversing construction.}
{We reverse the path} between two critical cells
  $$\beta=\big(\spadesuit \ \{k\}
I \ \clubsuit \{n,*,j\}\big) \hbox{  and  }\alpha=\big(\spadesuit \
\{k\} I \ \clubsuit \cup\{j\} \ \{n,*\}\big)$$

if and only if the three conditions hold:
\begin{enumerate}
  \item  $j>*$,
  \item $j>\clubsuit$,
  \item $j>k$.
\end{enumerate}

\bigskip

Let us first informally comment on the conditions (1)--(3).
The role of conditions (1) and (2) is to make the resulting vector field satisfy the first axiom.
Indeed, these two conditions imply that a critical cell participates in at most one reversed path.
The condition (3) is also important, but the reasons are less obvious:  the reversal of
all the paths  satisfying the first two conditions
yields a discrete vector field with closed gradient paths.

\bigskip

\begin{figure}\label{FigNewRules}
\centering
\includegraphics[width=14 cm]{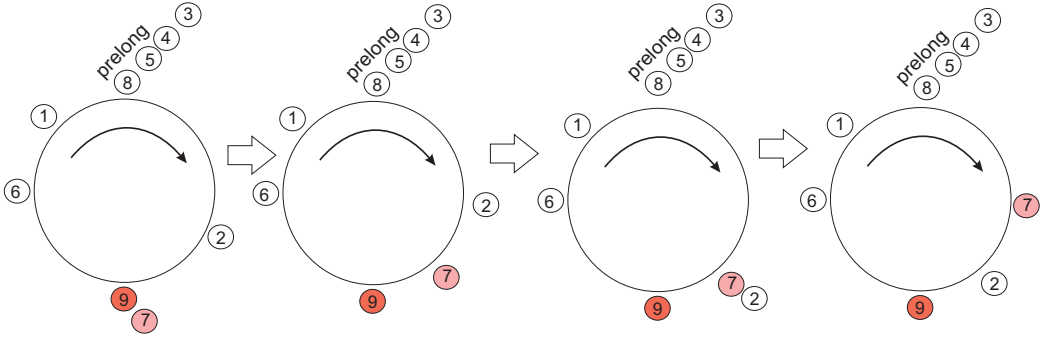}
\caption{An example of a path which is reversed}\label{FigFinalCrit}
\end{figure}

The \textbf{critical cells} that survive path reversing are
(See
Figure \ref{FigFinalCrit}):\begin{enumerate}
    \item All the cells of type 1, and

    \item  All the cells $\big(\spadesuit \ \{k\}
I \ \clubsuit \{n,*\}\big)$ of type 2 such that
$$k>*,\ \hbox{and} \   k>\clubsuit.$$
\end{enumerate}

\begin{figure}[h]
\centering
\includegraphics[width=10 cm]{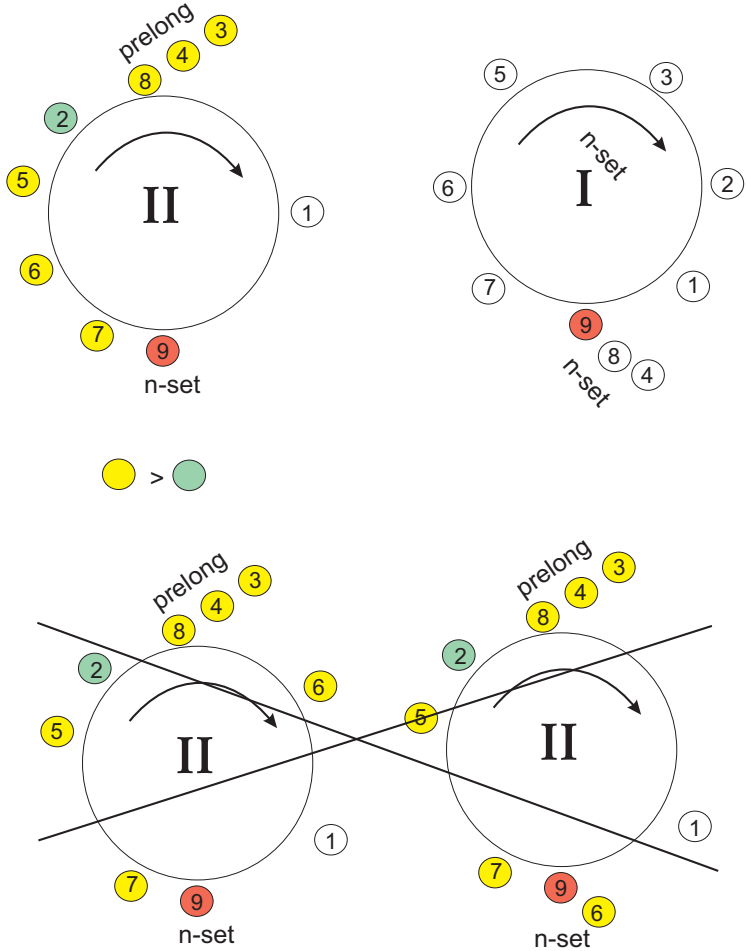}
\caption{Critical cells that survived the path reversing: examples
and non-examples}\label{FigFinalCrit}
\end{figure}

\begin{prop}
The above described path reversing  yields a discrete vector field.
\end{prop}
Proof. The second axiom of discrete vector field is straightforward. For the first axiom  observe that a cell of the complex participates in at most one gradient path that is reversed.\qed

\bigskip

We stress once again that unlike reversal of one single path, reversal of several gradient paths  does not automatically yield a discrete Morse function. So we have to prove the following proposition. 
\begin{prop}
The above described discrete vector field is a a discrete Morse function.
\end{prop}
Proof.  Assume the contrary: there exists a closed path $\Gamma$. It
can be decomposed into  some reversed and some unreversed gradient
paths between the (former) critical cells. Since a path from type 1
to type 2 never exists, we conclude that all (former) critical cells
that appear in $\Gamma$ are of type $2$. For them there are two
possibilities: either (1) all these former critical cells  have one
and the same entry $k$ preceding the prelong set, or (2) for some of
these (former critical) cells the entries preceding the prelong set
are different. We treat these cases separately.
\begin{enumerate}
\item   Lemma \ref{LemmaSplitPrelong} implies that the prelong set is maintained during the path.
Therefore no entry greater than $k$ passes  through the prelong
set. Also no entry {smaller}  than $k$  passes  through the
$n$-set. So, no entry makes a full turn.

The closed path $\Gamma$ necessarily includes a reversed path. This
means that at some moment, an entry $i$ greater than  $k$ comes from
$\clubsuit$ and joins the $n$-set. Consider the consecutive
split-step. \begin{enumerate}
    \item If some entry $j$ of the $n$-set moves forward, it never
    comes back.
    \item If some entry $j$ of the $n$-set moves backward, $j$ is
    necessarily smaller than  $k$, and the entry $j$ never comes
    back.
\end{enumerate}

\item Assume there are different entries right before the prelong sets in this path.
Let $j$ be the minimal of these entries. At some step of the path,
$j$  leaves the place before
 the prelong set. The entry $j$ is smaller than the next entry that gets to the place before the prelong set, so it can stay neither in
  $\spadesuit$ nor in the prelong set. Therefore, $j$ eventually  joins the $\clubsuit$. The only way for $j$
 to get back leads through the $n$-set, where it can get only via some reversed path.
 Since $j$ is minimal, during that path before the prelong set stands an entry greater
 than $j$, which is impossible, according to the reversing  condition (3).\qed
\end{enumerate}

\begin{thm}
The number of critical cells equals the sum of Betti numbers of the manifold
$M(L)$. Consequently, the above described pairing together with path
reversal gives a perfect Morse function.
\end{thm}
Proof. We know from \cite{faS}  that each short set containing the
entry $n$  contributes ''$2$'' to the sum of Betti numbers. So, to
prove the theorem,  we build a bijection between the short sets
containing $n$ and pairs of critical cells.

More precisely, we will show that for every  short set $J$ consisting of $k+1$ elements and containing
the entry   $n$ gives exactly  one $k$-cell of type 1,
and exactly one  $(n-3-k)$-cell of type 2.
\begin{enumerate}
\item \textbf{Cell of type 1}.
Take $J$ as the $n$-set of the (uniquely defined)  critical cell of type $1$.

Conversely, each critical cell of type $1$ gives a short set containing $n$, that is, the $n$-set.

\item \textbf{Cell of type 2.}

 \begin{enumerate}
    \item \textbf{Compose a prelong set $I$. }The set $\overline{J}:=[n]\setminus J$ is long. Take the largest entry of $\overline{J}$ and start
the prelong set $I$ with it. Keep  adding entries from
$\overline{J}$ to $I$ in the decreasing order as long as $I$ stays
short. The process stops once $I$ becomes prelong (that is, one step before it gets long).
    \item \textbf{Specify  an entry preceding $I$.} Let $j$ be the largest of the
 $\overline{J}\setminus I$. Turn $j$ to the singleton that precedes
the prelong set $I$.
    \item \textbf{Compose an $n$-set.} Define the $n$-set as $\big(\overline{J}\setminus (I\cup \{j\}\big)\big)\cup\{n\}$. By construction, each entry in
the $n$-set (except for $n$) is smaller than $j$. Clearly, we get a short set (since the complement is long). 
 \item \textbf{Positions of the rest of the singletons are now defined
 uniquely.} We turn all other entries to
singletons, which are placed before $\{j\}$, if they are
larger than $j$, and  after $I$ if they are smaller than $j$.
 \end{enumerate}

 Now  compute the number of the sets in the partition.
 All entries of $J$ except $n$ turn to singletons. Moreover, we have a singleton $j$ and two non-singleton sets. Altogether we have $k+3$ sets, so the dimension of the cell is
 $(n-3-k)$.

Conversely,  each
critical cell of type 2 of the new Morse function
arises in this way:
assume we have a critical cell of type $2$.
     Take all the singletons except for the singleton that
    precedes the prelong set. Add the entry $n$. Altogether they
    give the short set containing $n$ associated to the cell.

\qed
\end{enumerate}

\subsection*{Two examples}
Let $L=(1,1,1,1,1,1,1)$ be the equilateral 7-linkage.

\begin{enumerate}

\item For the short set  $J=\{7\}$, we have: \begin{enumerate}
\item The associated cell of type 1 is $\big(\{6\} \ \{5\} \ \{4\} \ \{3\} \ \{2\} \ \{1\} \ \{7\}\big);$
\item $\overline{J}=\{1,2,3,4,5,6\}$, $I=\{4,5,6\}$, $j=3$, and the associated cell of type 2 is:
$$\big(\{3\} \ \{4,5,6\} \ \{7,1,2\}\big)$$
\end{enumerate}
\item  For the short set  $J=\{5,6,7\}$, we have: \begin{enumerate}
\item The associated cell of type 1 is $\big(\{4\} \ \{3\} \ \{2\} \ \{1\} \ \{7,5,6\}\big);$
\item $\overline{J}=\{1,2,3,4\}$, $I=\{2,3,4\}$, $j=1$ and the associated cell of type 2 is:
$$(\{6\} \ \{5\} \ \{1\} \ \{2,3,4\} \ \{7\}\big)$$ 
\end{enumerate}

\end{enumerate}

\end{document}